\documentclass[12pt]{amsart}
\usepackage{amssymb}
\theoremstyle{plain}
 \newtheorem{thm}{Theorem}[section]
 
 \newtheorem{lem}[thm]{Lemma}

\theoremstyle{definition}

\theoremstyle{remark}
 \newtheorem{rem}{Remark}[section]
\begin{document}
\title[Non-commutative A-G mean inequality]
{Non-commutative A-G mean inequality}
\author[
Tomohiro Hayashi]{{Tomohiro Hayashi} }
\address[Tomohiro Hayashi]
{Nagoya Institute of Technology, 
Gokiso-cho, Showa-ku, Nagoya, Aichi, 466-8555, Japan}
\email[Tomohiro Hayashi]{hayashi.tomohiro@nitech.ac.jp}

\baselineskip=17pt
\maketitle

\begin{abstract}
In this paper we consider non-commutative 
analogue for the arithmetic-geometric mean inequality 
$$a^{r}b^{1-r}+(r-1)b\geq ra$$ 
for two positive numbers $a,b$ and $r> 1$. 
We show that under some assumptions the non-commutative 
analogue for $a^{r}b^{1-r}$ which satisfies this inequality 
is unique and equal to $r$-mean. 
The case $0<r<1$ is also considered. In particular, 
we give a new characterization of the geometric mean.

\end{abstract}

\section{Introduction}
For any two positive numbers $a,b$ 
and $r> 1$, 
we have the arithmetic-geometric mean inequality 
$$a^{r}b^{1-r}+(r-1)b\geq ra.$$ 
In this paper we consider the non-commutative 
analogue of this inequality for bounded 
linear operators on a Hilbert space. 
In particular, 
we give a new characterization of the geometric mean. 
Recently their ingenious paper~\cite{CL}, 
Carlen and Lieb used a certain non-commutative 
analogue of this inequality. 
Their paper is a motivation of our considerations.

There is one obvious 
non-commutative 
analogue as follows. For a bounded positive 
operator $X$ on a Hilbert space, we always have 
$$X^{r}+(r-1)\geq rX.$$
For any two positive invertible operators 
$A,B$, set 
$X=B^{-1/2}AB^{-1/2}$. 
Then we have 
$$
(B^{-1/2}AB^{-1/2})^{r}+(r-1)\geq rB^{-1/2}AB^{-1/2}
$$
and hence 
$$
B^{1/2}(B^{-1/2}AB^{-1/2})^{r}B^{1/2}+(r-1)B
\geq rA.
$$
Thus if we consider 
$B^{1/2}(B^{-1/2}AB^{-1/2})^{r}B^{1/2}$ 
(so-called $r$-mean) 
as non commutative analogue 
for $a^{r}b^{1-r}$, we get a desired inequality. 

We conjecture that there is 
no other example of non-commutative analogue 
for the above arithmetic-geometric mean inequality. 
The main result of this paper is as follows. 
We consider ``non-commutative analogue'' 
$M(A,B)$ for $a^{r}b^{1-r}$. More precisely 
$M$ is a two variable map and $M(A,B)$ is a 
positive invertible operator for 
any two positive invertible operators 
$A,B$. We assume that 
\begin{enumerate}
\item
$M(tA,B)=t^{r}M(A,B)$ for any positive number $t$,
\item
$M(A,B)^{-1}=M(A^{-1},B^{-1})$.
\end{enumerate}
For example, 
$$
A^{r/2}B^{1-r}A^{r/2},\ \ \ \  
B^{(1+2r)/2}(B^{6}A^{-2}B^{6})^{-r/2}B^{(1+2r)/2}
$$
satisfy these conditions. 
Under these assumptions, 
if the inequality 
$$M(A,B)\geq rA+(1-r)B$$
holds, 
then we will show that 
$$M(A,B)=B^{1/2}(B^{-1/2}AB^{-1/2})^{r}B^{1/2}.$$
Therefore in a certain sense 
our conjecture is true. 

Of course these two assumptions are too strong. 
For example, 
$$
(A^{3}+2B)^{2}A^{r/2}(A^{3}+2B)^{-2}
B^{1-r}(A^{3}+2B)^{-2}A^{r/2}
(A^{3}+2B)^{2}$$
can be considered as non-commutative analogue for 
$a^{r}b^{1-r}$. However this does not satisfy 
our assumptions. 

We shall also consider the case $0<r<1$ and 
show a similar result. 
That is, under the assumptions (i) and (ii), 
if the inequality 
$$M(A,B)\leq rA+(1-r)B$$
holds, 
then we will show that 
$$M(A,B)=B^{1/2}(B^{-1/2}AB^{-1/2})^{r}B^{1/2}.$$

Our result can be considered as a characterization of 
$r$-mean, 
in particular the geometric mean. 
In the paper~\cite{AN} T.~Ando and K.~Nishio 
gave a characterization of the harmonic mean. 

The author wishes to express his hearty gratitude to Professor 
Tsuyoshi Ando for valuable comments. 
The author is also grateful to Professor 
Yoshihiro Nakamura for discussion. 
The author would like to thank Professors 
Hideki Kosaki, 
Mitsuru Uchiyama 
and Atsushi Uchiyama for useful advice 
and comments.

\section{Main Result}
Throughout this paper we assume that 
the readers are familiar with 
basic notations and results on operator 
theory. We refer the readers to 
Conway's book~\cite{C}. 

We denote 
by ${\frak H}$ a 
(finite or infinite dimensional) 
complex Hilbert space 
and by $B({\frak H})$ 
all bounded linear operators on it. 
For each operator $A\in B({\frak H})$, 
its operator norm is denoted by 
$||A||$. 
We denote by $B({\frak H})^{+}$ the set 
of all positive invertible operators. 
For two vectors $\xi,\eta\in {\frak H}$, 
their inner product and norm are denoted by 
$\langle \xi,\eta\rangle$ and $||\xi||$ 
respectively. 

In this paper we consider the map $M(\cdot,\cdot)$ 
from $B({\frak H})^{+}\times B({\frak H})^{+}$ 
to $B({\frak H})^{+}$. 

We fix a positive number 
$r> 0$. For $A,B\in B({\frak H})^{+}$, 
define 
$$
M_{r}(A,B)=B^{1/2}(B^{-1/2}AB^{-1/2})^{r}B^{1/2}.
$$
Here we remark that 
$$
B^{1/2}(B^{-1/2}AB^{-1/2})^{r}B^{1/2}=
A^{1/2}(A^{-1/2}BA^{-1/2})^{1-r}A^{1/2}.\eqno{(\dagger)}$$
(This is well-known for specialists.) 
Indeed, if $r$ is an integer, direct computations 
show this equality. Thus for any polynomial $p(t)$ 
with $p(0)=0$, we have 
$$
B^{1/2}\cdot p(B^{-1/2}AB^{-1/2})\cdot B^{1/2}=
A^{1/2}\cdot p((A^{-1/2}BA^{-1/2})^{-1})
\cdot(A^{-1/2}BA^{-1/2})\cdot A^{1/2}.$$ 
Thus by continuity we get $(\dagger)$. 
The map $M_{r}$ is so-called $r$-mean, and usually 
the case $0<r<1$ is considered. 
(When $0<r<1$, $M_{r}(A,B)$ is one of the 
so-called operator means. In particular, 
in the case $r=1/2$, $M_{r}(A,B)$ 
is said to be the geometric mean.)

First we shall consider the case $r>1$. The following 
is our main result. 

\begin{thm}
Assume $r>1$. 
For any $A,B\in B({\frak H})^{+}$, if the map $M$ satisfies 
\begin{enumerate}
\item
$M(A,B)\geq rA+(1-r)B,$
\item
$M(tA,B)=t^{r}M(A,B)$ for any positive number $t$, 
\item
$M(A,B)^{-1}=M(A^{-1},B^{-1}),$
\end{enumerate}
then we have $M=M_{r}$.
\end{thm}

We need some preparations to prove this theorem. 
The following lemma states that 
under the assumptions (i) and (ii), 
the map $M_{r}(A,B)$ is 
``less'' than $M(A,B)$ in a certain sense. 
(See Remark 2.1.)

\begin{lem}
For any $A,B\in B({\frak H})^{+}$, we 
assume that the map $M$ satisfies 
\begin{enumerate}
\item
$M(A,B)\geq rA+(1-r)B,$
\item
$M(tA,B)=t^{r}M(A,B)$ for any positive number $t$.
\end{enumerate}
Then for any unit vector $\xi\in{\frak H}$, if 
$r\geq 2$ 
we have 
$$
\langle 
A^{-1/2}M(A,B)A^{-1/2}\xi,\xi
\rangle
\langle 
(A^{-1/2}M_{r}(A,B)A^{-1/2})^{-1}\xi,\xi
\rangle
\geq 1.
$$
On the other hand if $1< r\leq 2$ we have 
$$\langle (A^{-1/2}M(A,B)A^{-1/2})^{1/(r-1)}
\xi,\xi\rangle
\langle (A^{-1/2}M_{r}(A,B)A^{-1/2})^{-1/(r-1)}
\xi,\xi\rangle\geq 1.$$

\end{lem}

\begin{proof}
By assumptions we have 
$$t^{r}M(A,B)\geq rtA+(1-r)B$$
and hence 
$$A^{-1/2}M(A,B)A^{-1/2}\geq rt^{1-r}+
(1-r)t^{-r}A^{-1/2}BA^{-1/2}.$$
For a unit vector $\xi\in {\frak H}$, set 
$$f(t)=
rt^{1-r}
+(1-r)t^{-r}\langle A^{-1/2}BA^{-1/2}
\xi,\xi\rangle.$$
Here we remark that 
$\langle 
A^{-1/2}M(A,B)A^{-1/2}\xi,\xi
\rangle
\geq f(t)$. 
Then it is easy to see that 
the maximum value of $f(t)$ 
on $(0,\infty)$ is 
equal to 
$\langle A^{-1/2}BA^{-1/2}
\xi,\xi\rangle^{1-r}$. 
Thus we get 
$$
\langle A^{-1/2}M(A,B)A^{-1/2}
\xi,\xi\rangle
\langle A^{-1/2}BA^{-1/2}
\xi,\xi\rangle^{r-1}
\geq 1.$$

In the case $r\geq 2$, 
by the Jensen inequality and $(\dagger)$ we have 
$$\langle A^{-1/2}BA^{-1/2}
\xi,\xi\rangle^{r-1}
\leq 
\langle (A^{-1/2}BA^{-1/2})^{r-1}
\xi,\xi\rangle
=
\langle 
(A^{-1/2}M_{r}(A,B)A^{-1/2})^{-1}\xi,\xi
\rangle.
$$
So we are done. 

Next we consider the case $1< r \leq 2$. 
Let $s$ be a positive number 
such that 
$\frac{1}{r}+\frac{1}{s}=1$. 
Then since $(r-1)=1/(s-1)$, we have 
$$\langle A^{-1/2}M(A,B)A^{-1/2}
\xi,\xi\rangle
\langle A^{-1/2}BA^{-1/2}
\xi,\xi\rangle^{1/(s-1)}\geq 1$$
and hence 
$$\langle A^{-1/2}M(A,B)A^{-1/2}
\xi,\xi\rangle^{s-1}
\langle A^{-1/2}BA^{-1/2}
\xi,\xi\rangle\geq 1.$$
Since $s\geq 2$ 
and $(s-1)=1/(r-1)$, 
we compute as above 
\begin{align*}
\langle &(A^{-1/2}M(A,B)A^{-1/2})^{1/(r-1)}
\xi,\xi\rangle
\langle (A^{-1/2}M_{r}(A,B)A^{-1/2})^{-1/(r-1)}
\xi,\xi\rangle\\
&=
\langle (A^{-1/2}M(A,B)A^{-1/2})^{s-1}
\xi,\xi\rangle
\langle A^{-1/2}BA^{-1/2}
\xi,\xi\rangle\\
&\geq 
\langle (A^{-1/2}M(A,B)A^{-1/2})
\xi,\xi\rangle^{s-1}
\langle A^{-1/2}BA^{-1/2}
\xi,\xi\rangle
\geq 1.
\end{align*}
\end{proof}

\begin{thm}
For two positive invertible operators 
$X,Y\in B({\frak H})^{+}$, if 
they satisfy 
$$\langle 
X\xi,\xi
\rangle
\langle 
Y^{-1}\xi,\xi
\rangle
\geq 1$$ 
and 
$$\langle 
Y\xi,\xi
\rangle
\langle 
X^{-1}\xi,\xi
\rangle
\geq 1$$ 
for any unit vector 
$\xi\in {\frak H}$, 
then we have $X=Y$. 
\end{thm}

In order to show this theorem, we need the following 
lemma. 

\begin{lem}
For two operators $X,Y\in B({\frak H})^{+}$, 
they satisfy 
$$\langle 
X\xi,\xi
\rangle
\langle 
Y^{-1}\xi,\xi
\rangle
\geq 1$$ 
for any unit vector 
$\xi\in {\frak H}$ if and only if 
we have 
$$tX+(tY)^{-1}\geq 2$$
for any positive number $t$.

\end{lem}

\begin{proof}
Let $\xi\in {\frak H}$ be a unit vector. 
We set $f(t)=t\langle 
X\xi,\xi
\rangle
+t^{-1}
\langle 
Y^{-1}\xi,\xi
\rangle$ for $t>0$. 
Then it is easy to see that the minimum value of 
$f(t)$ is equal to 
$2\sqrt{
\langle 
X\xi,\xi
\rangle
\langle 
Y^{-1}\xi,\xi
\rangle}
$. Hence we are done.
\end{proof}

\begin{proof}[Proof of Theorem 2.3]
By the previous lemma we have 
$$tX+(tY)^{-1}\geq 2$$
and 
$$tY+(tX)^{-1}\geq 2$$
for any positive number $t$. 
Let $Z=Y^{1/2}XY^{1/2}$. 
Then we have 
$tZ+t^{-1}\geq 2Y$ and 
$t+(tZ)^{-1}\geq 2Y^{-1}$. 
So we get 
$$\dfrac{2tZ}{t^{2}Z+1}\leq Y
\leq 
\frac{t^{2}Z+1}{2t}.\eqno{(1)}
$$

First we assume that the Hilbert space is 
finite dimensional because 
in this case the proof becomes simpler. 
Take any projection $P$ of rank one which 
reduces $Z$, that is, 
$ZP=\lambda P$ for some positive number 
$\lambda$. 
(Here we use the fact that $Z$ is atomic, 
thanks to finite dimensionality.) 
Then we get 
$$\dfrac{2t\lambda}{t^{2}\lambda+1}P\leq PYP
\leq 
\frac{t^{2}\lambda+1}{2t}P.
$$
Since $P$ is of rank one, 
$PYP$ is of the form 
$PYP=\alpha P$ for some 
$\alpha>0$. 
Therefore taking the maximum in $t$ 
on the left-hand side and the minimum 
on the right-hand side we have 
$PYP=\lambda^{1/2}P$. 
On the other hand, since we also have 
$$
\dfrac{t^{2}Z+1}{2tZ}\geq Y^{-1}
\geq 
\frac{2t}{t^{2}Z+1},\eqno{(2)}
$$
we get 
$$
\dfrac{t^{2}\lambda+1}{2t\lambda}P\geq PY^{-1}P
\geq 
\frac{2t}{t^{2}\lambda+1}P
$$ 
and hence 
$PY^{-1}P=\lambda^{-1/2}P$ as above. 

Let $\xi\in {\frak H}$ be a vector satisfying 
$P\xi=\xi$. Then we have 
\begin{align*}
||Y^{1/2}\xi||\cdot||Y^{-1/2}\xi||&=
\langle PYP\xi,\xi\rangle^{1/2}
\langle PY^{-1}P\xi,\xi\rangle^{1/2}
=
\langle \lambda^{1/2}\xi,\xi\rangle^{1/2}
\langle \lambda^{-1/2}\xi,\xi\rangle^{1/2}\\
&=||\xi||^{2}
=\langle Y^{1/2}\xi,Y^{-1/2}\xi\rangle.
\end{align*}
By the equality condition for the
Cauchy-Schwarz inequality, 
this implies that $Y^{1/2}\xi$ is a 
scalar-multiple of $Y^{-1/2}\xi$, 
in other words, 
$Y\xi$ is a 
scalar-multiple of $\xi$. 
So we get $YP=PYP=Z^{1/2}P$. Since by the 
spectral theory for a positive matrix there are 
such projections $P_{i}$ of rank one such that 
$\sum_{i}P_{i}=1$, 
we conclude that $Y=Z^{1/2}=(Y^{1/2}XY^{1/2})^{1/2}$ 
and hence $X=Y$.

Next we shall consider the general case. 
The following argument is due to 
a private communication with 
T.~Ando~\cite{A2}. 
The author would like to thank Professor Ando for permitting 
the author to include his argument in this paper. 

It is easy to see that for any $t>0$ 
$$\dfrac{2tZ}{t^{2}Z+1}\leq Z^{1/2}
\leq 
\frac{t^{2}Z+1}{2t}
$$
and
$$
\dfrac{t^{2}Z+1}{2tZ}\geq Z^{-1/2}
\geq 
\frac{2t}{t^{2}Z+1}.
$$
Combining these with $(1)$ and $(2)$, we have 
$$\dfrac{2tZ}{t^{2}Z+1}
-\frac{t^{2}Z+1}{2t}
\leq Y-Z^{1/2}
\leq 
\frac{t^{2}Z+1}{2t}-\dfrac{2tZ}{t^{2}Z+1}
$$
and 
$$\dfrac{t^{2}Z+1}{2tZ}-\frac{2t}{t^{2}Z+1}
\geq Y^{-1}-Z^{-1/2}
\geq \frac{2t}{t^{2}Z+1}-
\dfrac{t^{2}Z+1}{2tZ}.
$$
Here we compute 
$$
\frac{t^{2}Z+1}{2t}-\dfrac{2tZ}{t^{2}Z+1}
=
(t-Z^{-1/2})^{2}
\dfrac{Z^{2}(t+Z^{-1/2})^{2}}{2t(t^{2}Z+1)}
$$
and 
$$
\frac{t^{2}Z+1}{2tZ}-\dfrac{2t}{t^{2}Z+1}
=
(t-Z^{-1/2})^{2}
\dfrac{Z(t+Z^{-1/2})^{2}}{2t(t^{2}Z+1)}.
$$
Therefore there is a positive number $\gamma$
such that 
for any spectrum $\lambda$ of $Z$ 
and a projection $P$ which reduces $Z$ 
we have 
$$
||PYP-Z^{1/2}P|| \leq \gamma||(\lambda-Z^{-1/2})P||^{2}
\eqno{(3)}
$$
and 

$$
||PY^{-1}P-Z^{-1/2}P||\leq \gamma||(\lambda-Z^{-1/2})P||^{2}.
\eqno{(4)}
$$

Let us use $(PY^{-1}P)^{-1}$ to denote the 
inverse of $PY^{-1}P$ on $P{\frak H}$. 
Then we see that 
$$
PY^{-1}P-Z^{-1/2}P=(PY^{-1}P)(Z^{1/2}P-(PY^{-1}P)^{-1})
Z^{-1/2}P
$$
and hence by using (4) 
there is a positive number $\gamma'$ such that 
$$
||(PY^{-1}P)^{-1}-Z^{1/2}P||\leq 
\gamma'||(\lambda-Z^{-1/2})P||^{2}.
$$
Combining this with (3) we conclude that 
there is a positive number $\gamma''$ such that 
for any spectrum $\lambda$ of $Z$ 
and spectral projection $P$ of $Z$ 
$$
||PYP-(PY^{-1}P)^{-1}||
\leq 
\gamma''||(\lambda-Z^{-1/2})P||^{2}.\eqno{(5)}
$$

For any integer $n$, 
take a partition 
of unity $\{P_{i}\}_{i=1}^{n}$ which consists of spectral 
projections of $Z$ such that 
there exist spectrums $\{\lambda_{i}\}_{i=1}^{n}$ of $Z^{-1/2}$
satisfying 
$$
||(\lambda_{i}-Z^{-1/2})P_{i}||
\leq 
\dfrac{||Z^{-1/2}||}{n}.
$$
Then it follows from (3) that 
$$
||\sum_{i=1}^{n}(P_{i}YP_{i}-Z^{1/2}P_{i})||\leq 
\dfrac{\gamma||Z^{-1/2}||^{2}}{n^{2}}.\eqno{(6)}
$$
Similarly it follows from (5) that 
$$
||P_{i}YP_{i}-(P_{i}Y^{-1}P_{i})^{-1}||
\leq 
\dfrac{\gamma''||Z^{-1/2}||^{2}}{n^{2}}.
$$
Recall the 
following formula, which is so-called 
Schur complement 
$$
(P_{i}Y^{-1}P_{i})^{-1}
=P_{i}YP_{i}-P_{i}YP_{i}^{\perp}
(P_{i}^{\perp}YP_{i}^{\perp})^{-1}
P_{i}^{\perp}YP_{i}
$$
where $P_{i}^{\perp}=1-P_{i}$. 
Indeed we can show  
\begin{align*}
\{P_{i}YP_{i}-&P_{i}YP_{i}^{\perp}
(P_{i}^{\perp}YP_{i}^{\perp})^{-1}
P_{i}^{\perp}YP_{i}\}\cdot P_{i}Y^{-1}P_{i}\\
&=P_{i}YP_{i}Y^{-1}P_{i}-
P_{i}YP_{i}^{\perp}
(P_{i}^{\perp}YP_{i}^{\perp})^{-1}
P_{i}^{\perp}YP_{i}Y^{-1}P_{i}\\
&=P_{i}YP_{i}Y^{-1}P_{i}-
P_{i}YP_{i}^{\perp}
(P_{i}^{\perp}YP_{i}^{\perp})^{-1}
(P_{i}^{\perp}Y-P_{i}^{\perp}YP_{i}^{\perp})
Y^{-1}P_{i}\\
&=P_{i}YP_{i}Y^{-1}P_{i}-P_{i}YP_{i}^{\perp}Y^{-1}P_{i}=P_{i}.
\end{align*}
By using this formula we have 
\begin{align*}
||P_{i}^{\perp}YP_{i}||^{2}
&=||(P_{i}^{\perp}YP_{i}^{\perp})^{1/2}
(P_{i}^{\perp}YP_{i}^{\perp})^{-1/2}
P_{i}^{\perp}YP_{i}||^{2}\\
&\leq 
||Y||\cdot||(P_{i}^{\perp}YP_{i}^{\perp})^{-1/2}
P_{i}^{\perp}YP_{i}||^{2}\\
&=||Y||\cdot
||P_{i}YP_{i}^{\perp}
(P_{i}^{\perp}YP_{i}^{\perp})^{-1}
P_{i}^{\perp}YP_{i}||\\
&=
||Y||\cdot||P_{i}YP_{i}-(P_{i}Y^{-1}P_{i})^{-1}||\\
&\leq \dfrac{\gamma''||Y||\cdot||Z^{-1/2}||^{2}}{n^{2}}. 
\end{align*}
Therefore for each unit vector 
$\xi\in {\frak H}$ by using the Cauchy-Schwarz inequality 
we see that 
\begin{align*}
||\sum_{i=1}^{n}
P_{i}^{\perp}YP_{i}\xi||
&\leq 
\sum_{i=1}^{n}
||P_{i}^{\perp}YP_{i}||\cdot||P_{i}\xi||\\
&\leq 
\sqrt{\sum_{i=1}^{n}
||P_{i}^{\perp}YP_{i}||^{2}}
\sqrt{\sum_{i=1}^{n}||P_{i}\xi||^{2}}\\
&=\sqrt{\sum_{i=1}^{n}
||P_{i}^{\perp}YP_{i}||^{2}}\\
&
\leq 
\sqrt{\sum_{i=1}^{n}\dfrac{\gamma''||Y||\cdot||Z^{-1/2}||^{2}}{n^{2}}}
=\sqrt{\dfrac{\gamma''||Y||\cdot||Z^{-1/2}||^{2}}{n}}.
\end{align*}
Thus we get 
$$
||\sum_{i=1}^{n}
P_{i}^{\perp}YP_{i}||
\leq 
\sqrt{\dfrac{\gamma''||Y||\cdot||Z^{-1/2}||^{2}}{n}}
.\eqno{(7)}
$$
By using (6) and (7) we see that 
\begin{align*}
||Y-Z^{1/2}||
&\leq 
||\sum_{i=1}^{n}P_{i}YP_{i}-Z^{1/2}P_{i}||
+||\sum_{i=1}^{n}
P_{i}^{\perp}YP_{i}||\\
&\leq 
\dfrac{\gamma||Z^{-1/2}||}{n^{2}}
+
\sqrt{\dfrac{\gamma''||Y||\cdot||Z^{-1/2}||^{2}}{n}}.
\end{align*}
By tending $n\rightarrow\infty$ we get $Y=Z^{1/2}$ 
and hence $X=Y$.

\end{proof}

Now we can prove our main result. 

\begin{proof}[Proof of Theorem 2.1]
First we consider the case $r\geq 2$. 
Set 
$$X=A^{-1/2}M(A,B)A^{-1/2}\ 
{\text{and}}\ 
Y=A^{-1/2}M_{r}(A,B)A^{-1/2}.$$ 
By Lemma 2.2 for any unit vector 
$\xi\in {\frak H}$ we have 
$$\langle 
X\xi,\xi
\rangle
\langle 
Y^{-1}\xi,\xi
\rangle\geq 1.
$$
On the other hand, 
thanks to the relations 
$M(A,B)^{-1}=M(A^{-1},B^{-1})$ 
and 
$M_{r}(A,B)^{-1}=M_{r}(A^{-1},B^{-1})$, 
applying Lemma 2,2 for the pair 
$(A^{-1},B^{-1})$ 
we have 
\begin{align*}
\langle 
X^{-1}&\xi,\xi
\rangle
\langle 
Y\xi,\xi
\rangle\\
&=
\langle 
A^{1/2}M(A^{-1},B^{-1})A^{1/2}\xi,\xi
\rangle
\langle 
(A^{1/2}M_{r}(A^{-1},B^{-1})A^{1/2})^{-1}\xi,\xi
\rangle
\geq 1.
\end{align*}
Therefore by Theorem 2.3 we get $X=Y$ and hence 
$M=M_{r}$. 

In the case $1< r \leq 2$, 
set 
$$X=(A^{-1/2}M(A,B)A^{-1/2})^{1/(r-1)}
\ {\text{and}}\ 
Y=(A^{-1/2}M_{r}(A,B)A^{-1/2})^{1/(r-1)}.$$ 
Then in the same way 
we conclude the desired fact.
\end{proof}

\begin{rem}
\begin{enumerate}
\item 
For positive invertible operators 
$A,B,C$, the block matrix
$$
\begin{pmatrix}
A&B\\
B&C
\end{pmatrix}
$$
is positive if and only if 
$A\geq BC^{-1}B$~\cite{A}. Therefore 
for two positive invertible operators 
$X,Y$, the block matrix
$$
\begin{pmatrix}
X&1\\
1&Y^{-1}
\end{pmatrix}
$$
is positive if and only if 
$X\geq Y$. On the other hand 
for any unit vector $\xi\in {\frak H}$ 
the matrix
$$
\begin{pmatrix}
\langle 
X\xi,\xi
\rangle&1\\
1&\langle Y^{-1}\xi,\xi
\rangle
\end{pmatrix}
$$
is positive if and only if 
$\langle 
X\xi,\xi
\rangle
\langle 
Y^{-1}\xi,\xi
\rangle
\geq 1$. 
Thus the condition 
$\langle 
X\xi,\xi
\rangle
\langle 
Y^{-1}\xi,\xi
\rangle
\geq 1$ 
is weaker than $X\geq Y$. 
We do not know whether the condition 
$\langle 
X\xi,\xi
\rangle
\langle 
Y^{-1}\xi,\xi
\rangle
\geq 1$ 
define new order $X``\geq'' Y$ or not. 
The author guess that this relation 
does not satisfy transitivity. 
Here we remark that if $X``\geq'' Y$, 
then we have $X^{2}``\geq'' Y^{2}$. 
Indeed if we have $\langle 
X\xi,\xi
\rangle
\langle 
Y^{-1}\xi,\xi
\rangle
\geq 1$, 
then we get 
$$
\langle 
X^{2}\xi,\xi
\rangle
\langle 
Y^{-2}\xi,\xi
\rangle
\geq 
\langle 
X\xi,\xi
\rangle^{2}
\langle 
Y^{-1}\xi,\xi
\rangle^{2}
\geq 
1
.$$ 
Thus this relation is not equivalent to 
usual order. Theorem 2.3 states that if we have 
$X``\geq'' Y$ and $Y``\geq'' X$, then 
we conclude $X=Y$ (reflexivity). 

\item 
We would like to conjecture that Theorem 2.1 holds 
by replacing the condition (iii) with 
$$({\rm{iii}})'\ \ \ \ M(A,B)=A^{r}B^{1-r}
{\text{\ \ \ \ if $A$ commutes with $B$.}}
$$ 

\end{enumerate}
\end{rem}

Finally we shall prove the analogue 
in the case $0<r<1$ 
for Theorem 2.1. 

\begin{thm}
Assume $0<r<1$. 
For any $A,B\in B({\frak H})^{+}$, if the map $M$ satisfies 
\begin{enumerate}
\item 
$M(A,B)\leq rA+(1-r)B,$
\item
$M(tA,B)=t^{r}M(A,B)$ for any positive number $t$, 
\item
$M(A,B)^{-1}=M(A^{-1},B^{-1}),$
\end{enumerate}
then we have $M=M_{r}$.
\end{thm}

\begin{proof}
The proof is essentially same as that of 
Theorem 2.1. So we would like to give the sketch 
of the proof. 

By assumptions 
for any positive number $t$ 
we have 
$$
M(A,B)\leq rt^{r-1}A+(1-r)t^{r}B
$$
and 
$$
M(A,B)^{-1}\leq rt^{1-r}A^{-1}+(1-r)t^{-r}B^{-1}.
$$
Set 
$$Y=B^{-1/2}M(A,B)B^{-1/2}\ 
{\text{and}}\ 
Z=B^{-1/2}AB^{-1/2}.$$ 
Then we have 
$$
\dfrac{t^{r}Z}{rt+(1-r)Z}
\leq 
Y 
\leq 
\dfrac{rZ+(1-r)t}{t^{1-r}}.
$$
Then by the almost same arguments as those in the proof of Theorem 
2.3, we can show $Y=Z^{r}$. 
\end{proof}

\end{document}